\documentclass{amsart}

\usepackage[margin=1in]{geometry}
\usepackage{latexsym}
\usepackage{float}
\usepackage{amssymb}
\usepackage[all]{xy}
\usepackage{epsfig}
\usepackage{color}
\usepackage{graphics}
\usepackage{hyperref}

\newtheorem{Thm}[equation]{Theorem}
\newtheorem{Lem}[equation]{Lemma}

\theoremstyle{remark}

     \title{On  the HOMFLY Polynomial of $4$-plat presentations of knots}

     \author{Bo-hyun Kwon}
     \address{Department of Mathematics, Oklahoma State University, 401 Mathematical Sciences Stillwater, OK 74074, USA}
     \email{bohykwon@math.okstate.edu}
     \thanks{}

     \date{June, 13, 2014}


\begin{document}

\begin{abstract}
In this paper, I give a method to calculate the HOMFLY polynomials of two bridge knots by using a representation of the braid group $\mathbb{B}_4$ into a group of $3\times 3$ matrices. Also, I will give examples of a 2-bridge knot and a 3-bridge knot that have the same Jones polynomial, but different HOMFLY polynomials.
\end{abstract}

  \maketitle

\section{Introduction}

In 1985, Hoste-Ocneanu-Millett-Freyd-Lickorish-Yetter~\cite{3} had discovered the HOMFLY polynomial which is a 2-variable oriented link polynomial $P_L(a,m)$ motivated by the Jones polynomial. Also,  Prztycki and Traczyk~\cite{7} independently had done some work related to the HOMFLY polynomial.
The calculation of the HOMFLY polynomial is based on the HOMFLY skein relations as follows.\\

\begin{enumerate}

\item $P(L)$ is an isotopy invariant.

\item $P$(unknot)=1.

\item $a \cdot P(L_+)+a^{-1}\cdot P(L_{-})+m\cdot P(L_0)=0$.
\end{enumerate}

\begin{figure}[htb]
\includegraphics[scale=.4]{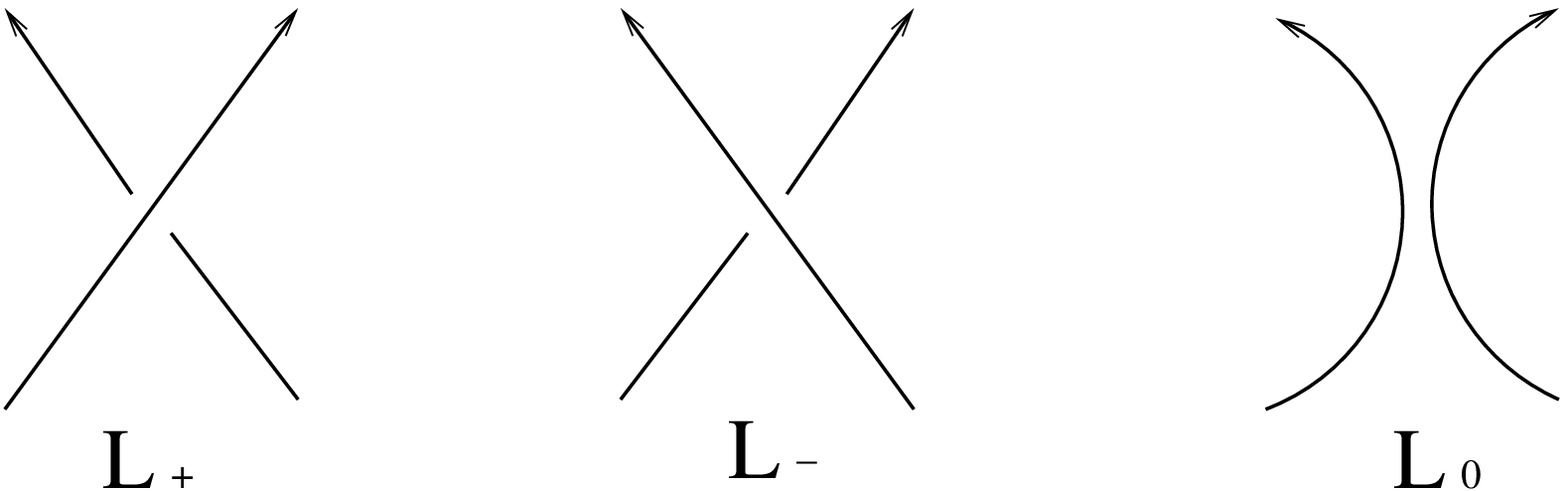}
\caption{}
\label{p1}
\end{figure}

Lickorish and  Millett~\cite{6} and Kanenobu and Sumi~\cite{8} gave a formula to calculate the HOMFLY polynomials of 2-bridge knots by using a representation of the continued fraction of a rational knot into a group of $2\times 2$ matrices. (See Proposition 14 of~\cite{6}.) \\

In this paper, I use the HOMFLY blacket polynomial and the plat presentation of knots to calculate the HOMFLY polonomials of rational knots by using a representation of the braid group $\mathbb{B}_4$ into a group of $3\times 3$ matrices. Also, we can extend the method to evaluate the HOMFLY polynomials of $2n$-plat presentations of knots.\\

Now, we define the plat presentation of a knot and a rational tangle.
Let $S^2$ be a sphere smoothly embedded in $S^3$ and let $K$ be a link transverse to $S^2$. The complement in $S^3$ of  $S^2$ consists of two open balls, $B_1$ and $B_2$. We assume that $S^2$ is $xz$-plane $\cup~\{\infty\}$.  Now, consider the projection of $K$ onto the flat $xy$-plane.
Then, the projection onto the $xy$-plane of $S^2$ is the $x$-axis and $B_1$ projects to the upper half plane and $B_2$ projects to the lower half plane.
The projection gives us a {\emph{ link diagram}}, where we make note of over and undercrossings.
The diagram of the link $K$  is called a $\emph{plat on 2n-strings}$, denoted by $p_{2n}(w)$,  if it is the union of a $2n$-braid $w$ and $2n$ unlinked and unknotted arcs which connect pairs of consecutive strings of the braid at the top and at the bottom endpoints and $S^2$ meets the top of the $2n$-braid. (See  the first and second diagrams of Figure~\ref{p2}.) Any link $K$ in $S^3$ admits a plat presentation, that is $K$ is ambient isotopic to a plat (\cite{2}, Theorem 5.1). The bridge (plat) number $b(K)$ of $K$ is the smallest possible number $n$ such that there exists a plat presentation of $K$ on $2n$ strings.\\

 We know that the braid group $\mathbb{B}_{4}$ is generated by $\sigma_1,\sigma_2,\sigma_{3}$ which are twisting of two adjacent strings. For example, $w=\sigma_2^{-2}\sigma_1^{2}\sigma_2^{-1}\sigma_3^2\sigma_2^{-1}$ is the word for the
4 braid of the first diagram of Figure~\ref{p2}.\\

\begin{figure}[htb]
\includegraphics[scale=.40]{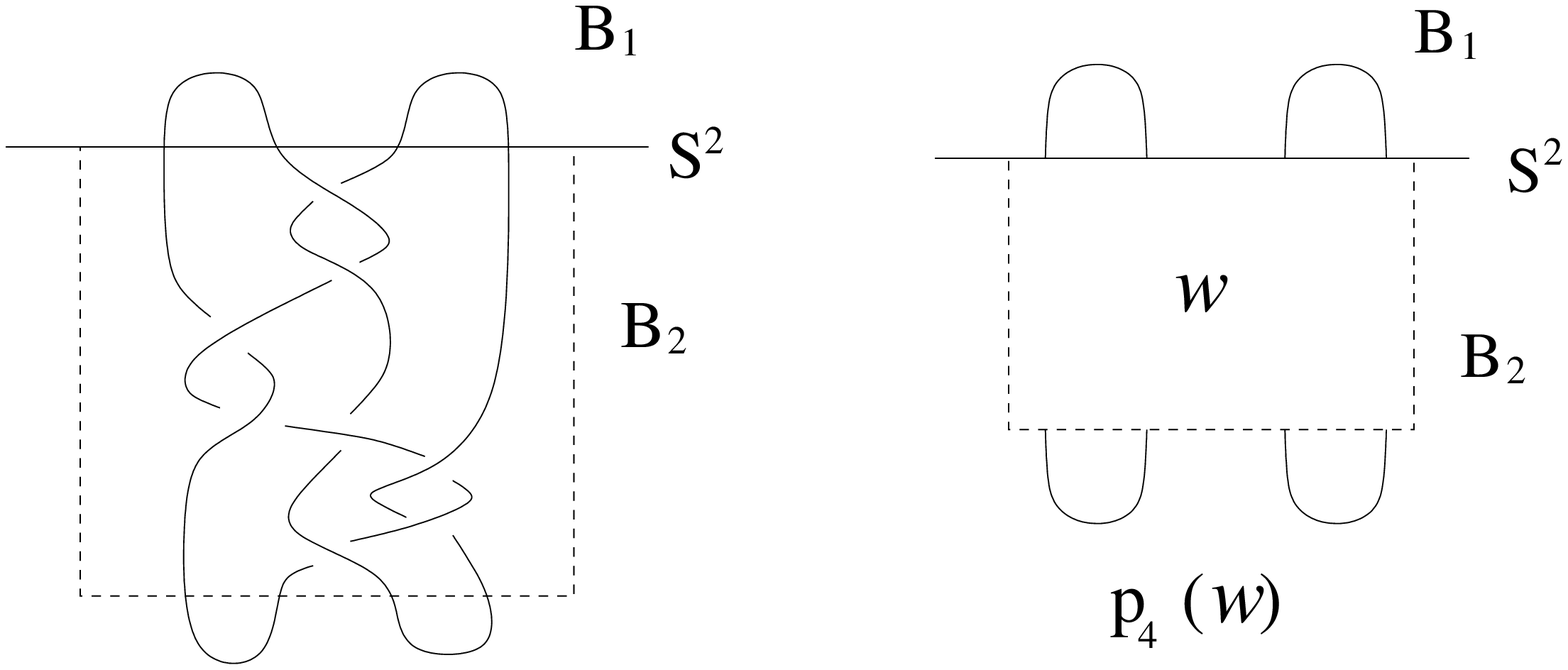}
\caption{}
\label{p2}
\end{figure}

 Then we say that a plat presentation  is $\emph{standard}$ if the $4$-braid $w$ of $p_{4}(w)$  involves only $\sigma_2,\sigma_3$.\\
 
 A $\emph{2-tangle}$ is the disjoint union of $2$ properly embedded arcs in a 3-ball $B^3$. 
 Then we say that a $2$-$tangle$ $T=(B^3,\alpha_1 \cup\alpha_2)$ is $\emph{rational}$ if there exists a homeomorphism of pairs ${H}: (B^3,\alpha_1 \cup\alpha_2)\longrightarrow
(D^2\times I,\{p_1,p_2\}\times I)$, where $p_i$ are distinct points in $D^2$ and $I=[0,1]$. We  have the two  $2$-tangles $T_1=(B_1,K\cap B_1)$ and $T_2=(B_2,K\cap B_2)$. We note that $T_1$ and $T_2$ are rational if $K$ has a plat presentation. Let $T_w$ be the rational $2$-tangle in $B_2$ if $K$ has a plat presentation.\\

Now, we define a $\emph{plat presentaion}$ for rational $2$-tangles $p_{4}(w)\cap B_2$ (Refer to~\cite{4}.) denoted by $q_{4}(w)$  if it is  the union of a $4$-braid $w$ and $2$ unlinked and unknotted arcs which connect pairs of strings of the braid  at the bottom endpoints with the same pattern as in a plat presentation  for a knot  and $\partial B_2$ meets the top of the $4$-braid.\\

We note that $q_{4}(w)$ is  a rational $2$-tangle in $B_2$. (See~\cite{5}.)\\

We say that $\overline{q_{4}(w)}$ = $p_{4}(w)$ is the $\emph{plat closure}$ of $q_{4}(w)$.\\

The tangle diagrams with the circles in Figure~\ref{p3} give the  diagrams of trivial rational $2$-tangles as in~\cite{1},~\cite{3},~\cite{4},~\cite{7}.

\begin{figure}[htb]
\includegraphics[scale=.32]{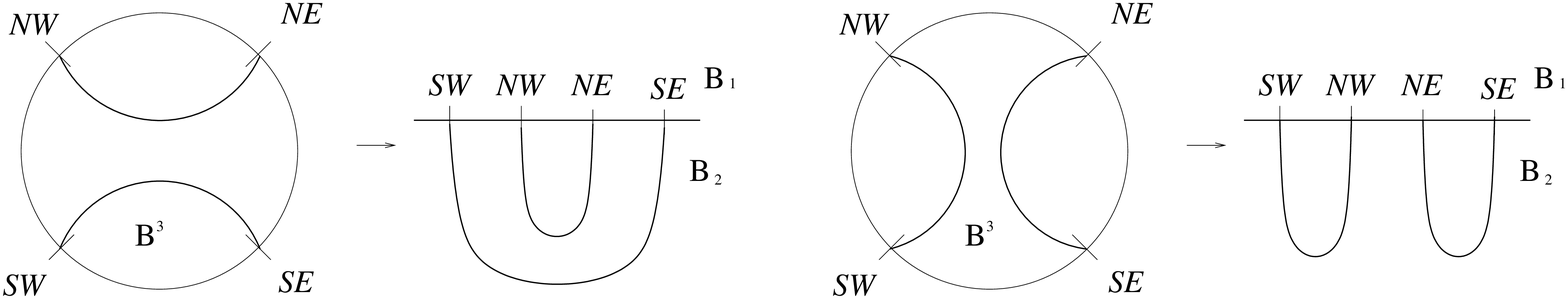}
\caption{}
\label{p3}
\end{figure}

We note that $q_4(w)$ is alternating if and only if $\overline{q_4(w)}$ is alternating.\\

A tangle $T$ is  $\emph{reduced}$ alternating if $T$ is alternating and $T$ does not have a self-crossing which can be removed by a Type I Reidemeister move.

\begin{Thm}[\cite{5}]\label{T1}
If $K$ is a 2-bridge knot, then there exists a word $w$ in $\mathbb{B}_4$ so that the  plat presentation $p_4(w)$ is reduced alternating, standard and represents a knot isotopic to $K$.
\end{Thm}

By Theorem~\ref{T1}, any 2-bridge knot $K$ can be represented by a word $w$ as a  plat which involves only $\sigma_2$ and $\sigma_3^{-1}$ (or $\sigma_2^{-1}$ and $\sigma_3$). i.e., $w=\sigma_2^{\epsilon_1}\sigma_3^{-\epsilon_2}\cdot\cdot\cdot\sigma_2^{\epsilon_{2n-1}}$
for some positive (negative) integers $\epsilon_i$ for $1\leq i\leq 2n-1$. We notice that if $w=\sigma_2^{\epsilon_1}\sigma_3^{-\epsilon_2}\cdot\cdot\cdot\sigma_3^{-\epsilon_{2n}}$  for some positive (negative) integer $\epsilon_{2n}$ then it is not a reduced alternating form. i.e., we can twist the right unlinked and unknotted bottom arc to reduce some crossings to have fewer crossings for $K$.
So, in order to have a reduced alternating form, $w$ needs to start from $\sigma_2^{\pm 1}$ and end at $\sigma_2^{\pm 1}$.\\

In section 2, we introduce the HOMFLY  polynomial of rational 2-tangles and  give a formula to calculate the HOMFLY polynomial of 4-plat presentations of knots.\\

In section 3, we give a method to find  the orientation of each crossing of a knot from a given orientation of the knot.\\

Then, we give some examples of knots for which we calclulate the HOMFLY polynomials and especially give examples of a 2-bridge knot and a 3-bridge knot that have the same Jones polynomial, but different HOMFLY polynomials in section 4.

\section{HOMFLY  bracket polynomial of  rational 2-tangles and the main theorem}

Let $K$ be a 2-bridge knot. By Theorem~\ref{T1} there exists a  plat presentation $p_4(w)$ which is reduced alternating, standard and represents a knot isotopic to $K$.\\

For a given orientation of $K$, we will give an induced orientation to the rational $2$-tangle $T_w=K\cap B_2$ such that $T_w$ has the same orientation with the oriented knot $\overrightarrow{K}$ in $B_2$. Let 
$\overrightarrow{q_4(w)}$ be the plat presentation of $T_w$ with the induced orientation.\\
 
 Now, we define the HOMFLY  polynomial of an oriented plat presentation of rational $2$-tangle $\overrightarrow{q_4(w)}$ in $B_2$ as $P(\overrightarrow{T_w})=f(a,m)<T_0>+g(a,m)<T_{\infty}>+h(a,m)<T_x>$, where the coefficients $f(a,m),g(a,m)$ and $h(a,m)$ are polynomials in $a,a^{-1}$ and $m$ that are obtained by starting with $\overrightarrow{T_w}$ and using the skein relations repeatedly until only the three tangles in the expression given for $T_w$.\\
 
 \begin{figure}[htb]
\includegraphics[scale=.4]{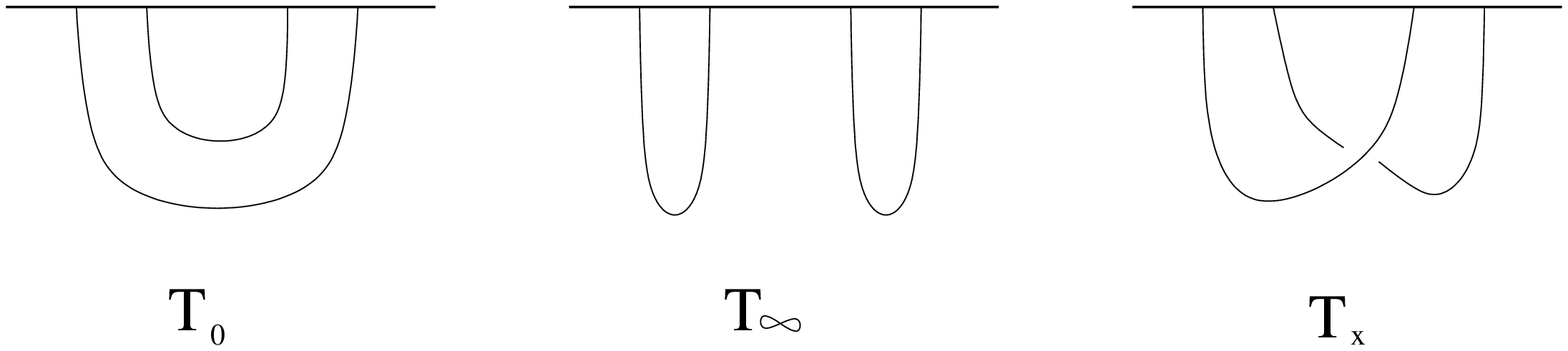}
\caption{}
\label{p4}
\end{figure}

 Let $n$ be the number of crossings of $w$.\\
 
 We note that  if we switch one of the alternative crossings of $w$ $(n>1)$ from positive (negative) to negative (positive) to have $K'$, then we  have the plat presentation $p_4(w')$ for $K'$ so that $w'$ is a reduced alternating, standard and $w'$  has lower crossings than $w$. (Refer to~\cite{5}.) So, by the skein relations, we can reduce the number of crossings of $w$.  However, if we have an oriented rational $2$-tangle with $n=1$ then we cannot reduce the number of crossing by the skein relations. This is the reason that we need $T_x$.\\
 
We remark that polynomials $f(a,m),g(a,m)$ and $h(a,m)$ are invariant under isotopy of $\overrightarrow{q_4(w)}$.\\

Also, we note that even if we apply  the skein relationship to one of the crossings of ${q_4(w)}$, the orientations of the rest of crossings will be preserved. i.e., all the rest of crossings keep the directions for the given orientation while we are applying the skein relationship to one of the crossings of $\overrightarrow{q_4(w)}$ to calculate the HOMFLY polynomial of ${q_4(w)}$.\\

Let $\mathcal{A}=<T_0>$, $\mathcal{B}=<T_{\infty}>$ and $\mathcal{C}=<T_x>$.\\

Recall that $K$ is alternating and standard.\\

Since $p_4(w)$ is standard, we consider $\mathbb{B}_3$ instead of $\mathbb{B}_4$. Then let $\sigma_1$ and $\sigma_2$ be the two generators of $\mathbb{B}_3$. I want to emphasize here that we are changing from $\sigma_2$ and $\sigma_3$ to $\sigma_1$ and $\sigma_2$.\\

Suppose that $w=\sigma_1^{\alpha_1}\sigma_2^{-\alpha_2}\cdot\cdot\cdot\sigma_1^{\alpha_{2n-1}}$ for positive  integers $\alpha_i$ ($1\leq i\leq 2n-1$).

Then, we will give an orientation to $w$ which is induced by $\overrightarrow{K}$. So, $\sigma_1$ and $\sigma_2^{-1}$ have four different cases $\sigma_{j1}$ and $\sigma^{-1}_{j2}$ for $j=1,2$ as in Figure~\ref{p5}.

\begin{figure}[htb]
\includegraphics[scale=.4]{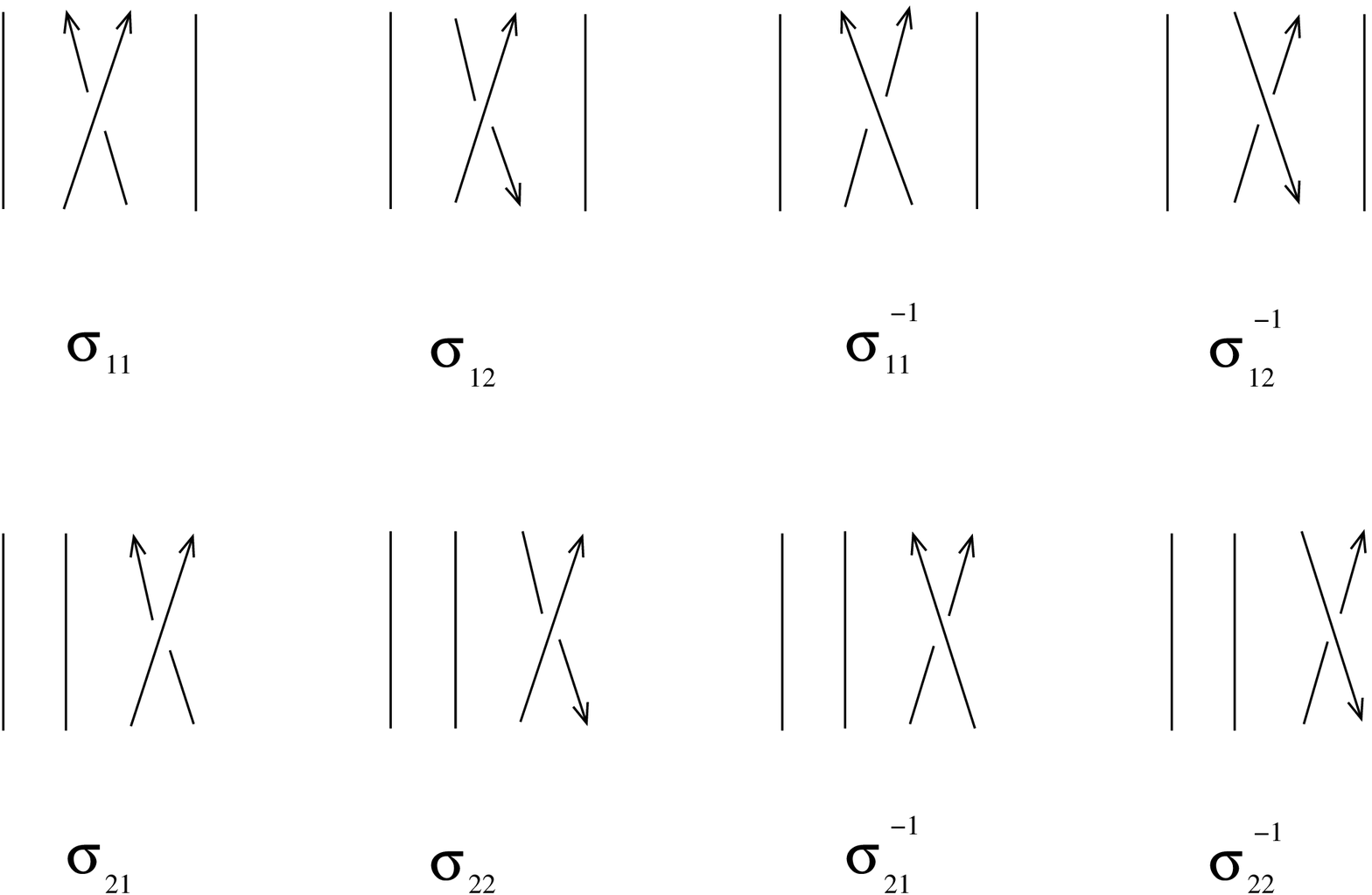}
\caption{}
\label{p5}
\end{figure}

Actually, there are two possible directions  for the orientation of $\overrightarrow{K}$. However, the skein relation does not depend on the choice of the direction.
Now, we consider the sub-directions which is induced by the orientation of $\overrightarrow{K}$ as in Figure~\ref{p5}. We note that there are the other corresponding eight cases which are obtained from the given cases  by taking the opposite arrows. However, we will not distinguish the corresponding cases since they play the same role as the corrsponding cases when we construct $3\times 3$ matrices for calculation of the HOMFLY polynomial. \\

So, by considering the orientation, we can describe $w$ as $\sigma_{1~k_{1}}^{\alpha_1}\sigma_{2~k_{2}}^{-\alpha_2}\cdot\cdot\cdot \sigma_{1~k_{2n-1}}^{\alpha_{2n-1}}$ instead of $w=\sigma_1^{\alpha_1}\sigma_2^{-\alpha_2}\cdot\cdot\cdot\sigma_1^{\alpha_{2n-1}}$, where $k_{i}\in\{1,2\}$.\\

Let $A^1_1= \left[ \begin{array}{ccc}
1 & 0 & 0 \\
0 & 0 & -a^{-2} \\
0 & 1& -a^{-1}m\\
 \end{array} \right]$, \hskip 50pt $A^1_2=\left[ \begin{array}{ccc}
1 & 0 & -am \\
0 & 0 & -a^2 \\
0 & 1 & 0\\
 \end{array} \right],\\$
 \vskip 20pt
 and,  $B^{-1}_1= \left[ \begin{array}{ccc}
0 & 0 & -a^2 \\
0 & 1 & 0 \\
1 & 0 & -am\\
 \end{array} \right]$,\hskip 50pt $B^{-1}_2= \left[ \begin{array}{ccc}
0 & 0 & -a^{-2} \\
0 & 1 & -a^{-1}m \\
1  & 0 & 0\\
 \end{array} \right].\\$
 \vskip 10pt
 Let $M=(A^1_{k_1})^{\alpha_1}(B^{-1}_{k_2})^{\alpha_2}\cdot\cdot\cdot(A^1_{k_{2n-1}})^{\alpha_{2n-1}}.$\\

\vskip 20pt
Then we have the main theorem to calculate the HOMFLY polynomial of $K$ as follows.\\

\begin{Thm}\label{T2}
Suppose that $q_4(w)$ is a  plat presentation of a rational 2-tangle $T_w$ which is alternating and standard so that
$w=\sigma_1^{\epsilon_1}\sigma_2^{-\epsilon_2}\cdot\cdot\cdot\sigma_1^{\epsilon_{2n-1}}$ for some positive  integers $\epsilon_i$ ($1\leq i\leq 2n-1$). Then
$P(T_w)=f(a,m)\mathcal{A}+g(a,m)\mathcal{B}+h(a,m)\mathcal{C}$, where $f(a,m), g(a,m)$ and $h(a,m)$ are obtained from
$[f(a,m)~g(a,m)~h(a,m)]=[0~1~0]M^t$. i.e., the second column of $M$.

Moreover, $P(K)=f(a,m)-g(a,m){(a+a^{-1})\over m}+h(a,m)$.
\end{Thm}

\begin{proof}
By Theorem~\ref{T1}, for a two bridge knot $K$, there exists a word $w$ in $\mathbb{B}_4$ so that the plat presentation $p_4(w)$ is alternating, standard and represents a link isotopic to $K$.\\

Then, by the argument above, we have $P(T_w)=f(a,m)\mathcal{A}+g(a,m)\mathcal{B}+h(a,m)\mathcal{C}$ for some 2-variable polynomials $f(a,m), g(a,m)$ and $h(a,m)$.\\

Since $w$ is standard, we consider two generators $\sigma_1$ and $\sigma_2$ for $\mathbb{B}_3$ as mentioned before.\\

Let $T_1'$ and $T_1''$ be the rational two tangles which are obtained from $T_w$ by adding $\sigma_1^{\mp 1}$ or $\sigma_2^{\pm 1}$ respectively to cancel the first $\sigma_i^{\pm 1}$ in $w$. So, we have a new word $v$ of smaller length than $w$ so that $w=\sigma_i^{\pm }v$ for the rational 2-tangles $T_1'$ and $T_1''$. Without loss of generality, we will consider the cases that $w=\sigma_1^{-1}v$ or $w=\sigma_2v$.\\

First, we consider the case that $w=\sigma_1^{-1}v$.
Then, we have that $P(T_1')=f'(a,m)\mathcal{A}+g'(a,m)\mathcal{B}+h'(a,m)\mathcal{C}$ for some 2-variable polynomials $f'(a,m),g'(a,m)$ and $h'(a,m)$.\\

Also, by Figure~\ref{p6}, we know that $P(T_1')=f(a,m)\mathcal{A}+g(a,m)<T_{x'}>+h(a,m)\mathcal{B}$, where $T_{x'}$ is the tangle which has a plat presentation $q_4(w)$ with $w=\sigma_1^{-1}$.\\

\begin{figure}[htb]
\includegraphics[scale=.35]{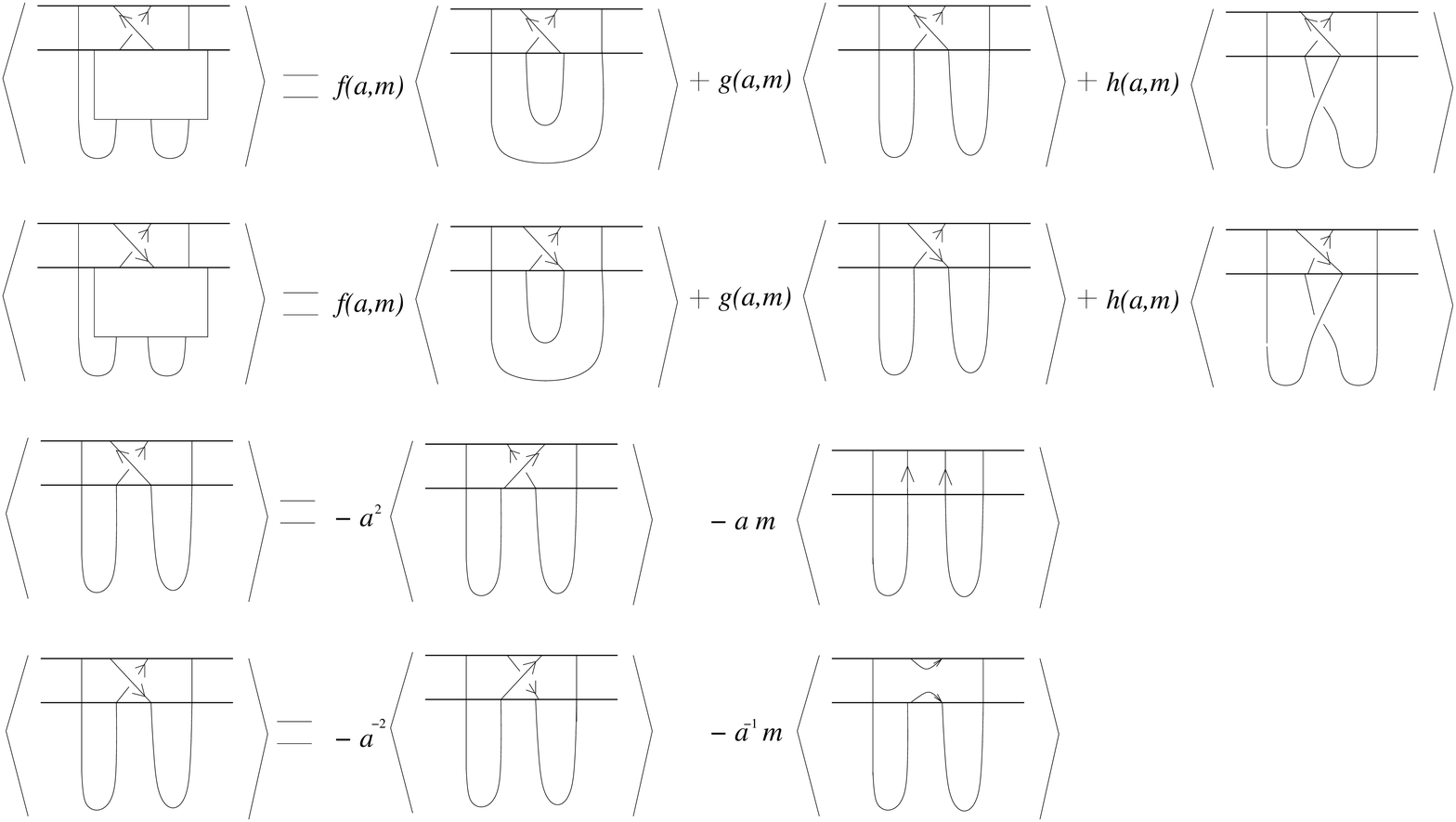}
\caption{}
\label{p6}
\end{figure}

By Figure~\ref{p6}, we also know that $P(T_{x'})=-a^2\mathcal{C}-am\mathcal{B}$ or $P(T_{x'})=-a^{-2}\mathcal{C}-a^{-1}m\mathcal{A}$.\\

Therefore, $P(T_1')=f(a,m)\mathcal{A}+g(a,m)<T_{x'}> +h(a,m)\mathcal{B}=f(a,m)\mathcal{A}+g(a,m)(-a^2\mathcal{C}-am\mathcal{B} ) +h(a,m)\mathcal{B}=f(a,m)\mathcal{A}+(h(a,m)-amg(a,m))\mathcal{B}-a^2g(a,m)\mathcal{C}$ or,\\

$P(T_1')=f(a,m)\mathcal{A}+g(a,m)<T_{x'}> +h(a,m)\mathcal{B}=f(a,m)\mathcal{A}+g(a,m)(-a^{-2}\mathcal{C}-a^{-1}m\mathcal{A} ) +h(a,m)\mathcal{B}=(f(a,m)-a^{-1}mg(a,m))\mathcal{A}+h(a,m)\mathcal{B}-a^{-2}g(a,m)\mathcal{C}$\\

Therefore, the following gives the operations.

$ \left[ \begin{array}{ccc}
1 & 0 & 0 \\
0 & -am & 1 \\
0 & -a^{2} & 0\\
 \end{array} \right]\left[\begin{array}{c}
f(a,m)\\
g(a,m)\\
h(a,m)\\
\end{array} 
  \right]=\left[\begin{array}{c}
f'(a,m)\\
g'(a,m)\\
h'(a,m)\\
\end{array} 
  \right],
   \left[ \begin{array}{ccc}
1 & -a^{-1}m & 0 \\
0 & 0 & 1 \\
0 & -a^{-2} & 0\\
 \end{array} \right]\left[\begin{array}{c}
f(a,m)\\
g(a,m)\\
h(a,m)\\
\end{array} 
  \right]=\left[\begin{array}{c}
f'(a,m)\\
g'(a,m)\\
h'(a,m)\\
\end{array} 
  \right].$
  
  \vskip 20pt
Now, we consider the case that $w=\sigma_2v$.\\
 
 Then we have $P(T_1'')=f''(a,m)\mathcal{A}+g''(a,m)\mathcal{B}+h''(a,m)\mathcal{C}$.\\

 By Figure~\ref{p7}, we also know that  $P(T_{x'})=-a^{-2}\mathcal{C}-a^{-1}m\mathcal{A}$ or $P(T_{x'})=-a^2\mathcal{C}-am\mathcal{B}$.\\

Therefore, $P(T_1')=f(a,m)<T_{x'}>+g(a,m)\mathcal{B}+h(a,m)\mathcal{C}=f(a,m)(-a^{-2}\mathcal{C}-a^{-1}m\mathcal{A} )+g(a,m)\mathcal{B} +h(a,m)\mathcal{A}=(h(a,m)-a^{-1}mf(a,m))\mathcal{A}+g(a,m)\mathcal{B}-a^{-2}f(a,m)\mathcal{C}$ or,\\

 $P(T_1')=f(a,m)<T_{x'}>+g(a,m)\mathcal{B}+h(a,m)\mathcal{C}=f(a,m)(-a^{2}\mathcal{C}-am\mathcal{B} )+g(a,m)\mathcal{B} +h(a,m)\mathcal{A}=h(a,m)\mathcal{A}+(g(a,m)-amf(a,m))\mathcal{B}-a^2f(a,m)\mathcal{C}$\\
 
  \begin{figure}
 \includegraphics[scale=.35]{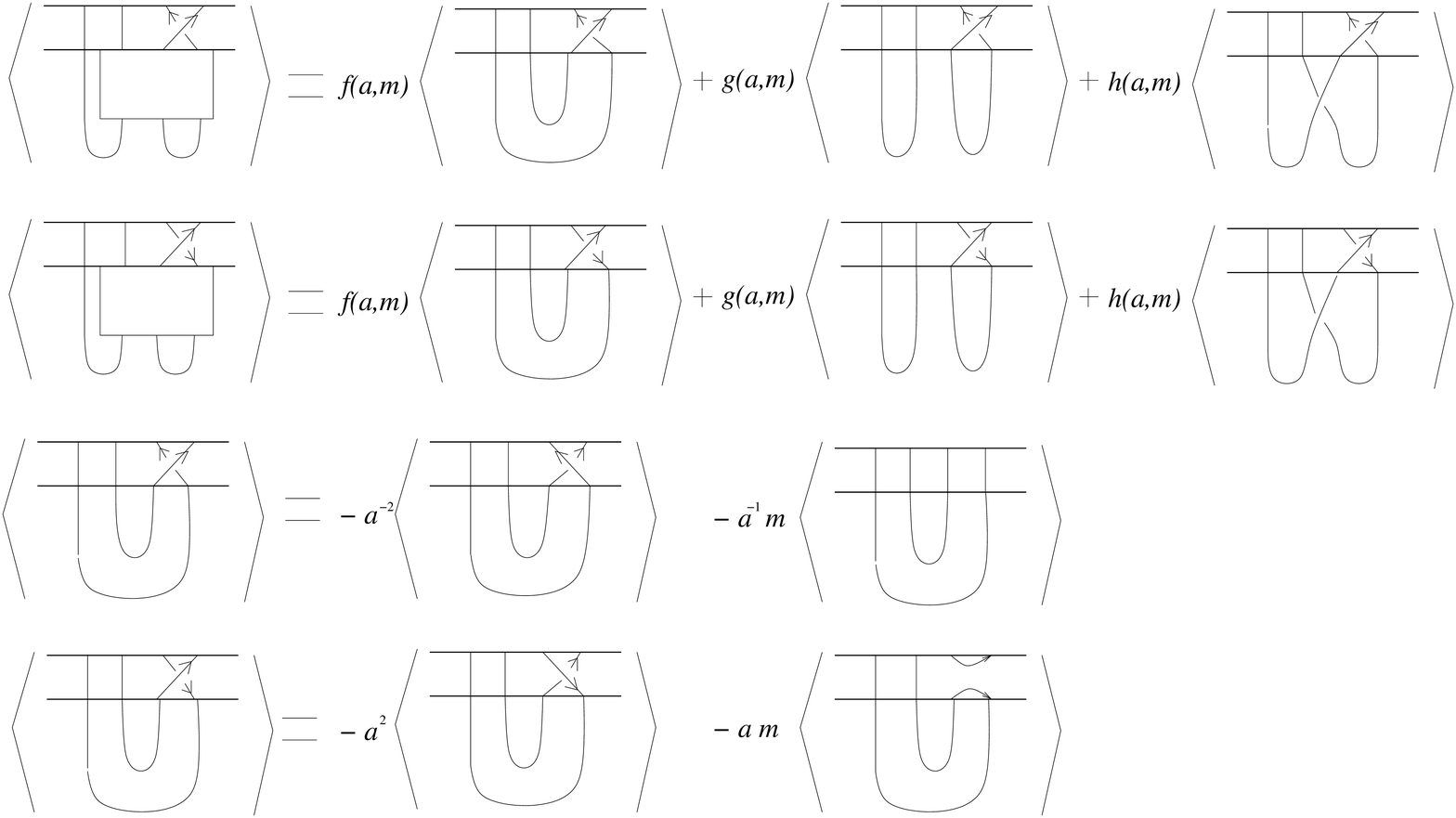}
 \caption{}
 \label{p7}
 \end{figure}
 
 Therefore, the following gives the operations.\\
  
 $ \left[ \begin{array}{ccc}
-a^{-1}m & 0 & 1 \\
0 & 1 & 0 \\
-a^{-2} & 0 & 0\\
 \end{array} \right]\left[\begin{array}{c}
f(a,m)\\
g(a,m)\\
h(a,m)\\
\end{array} 
  \right]=\left[\begin{array}{c}
f''(a,m)\\
g''(a,m)\\
h''(a,m)\\
\end{array} 
  \right],
   \left[ \begin{array}{ccc}
0 & 0 & 1 \\
-am & 1 & 0 \\
-a^{2} & 0 & 0\\
 \end{array} \right]\left[\begin{array}{c}
f(a,m)\\
g(a,m)\\
h(a,m)\\
\end{array} 
  \right]=\left[\begin{array}{c}
f''(a,m)\\
g''(a,m)\\
h''(a,m)\\
\end{array} 
  \right].$\\
 
\vskip 15pt
Now,  let $A^{-1}_1= \left[ \begin{array}{ccc}
1 & 0 & 0 \\
0 & -am & 1 \\
0 & -a^2 & 0\\
 \end{array} \right],$ \hskip 50pt $A^{-1}_2= \left[ \begin{array}{ccc}
1 & -a^{-1}m & 0 \\
0 & 0 & 1 \\
0 & -a^{-2} & 0\\
 \end{array} \right],\\$
 \vskip 20pt
 and,  $B^{1}_1= \left[ \begin{array}{ccc}
-a^{-1}m & 0 & 1 \\
0 & 1 & 0 \\
-a^{-2} & 0 & 0\\
 \end{array} \right]$,\hskip 50pt $B^{1}_2=\left[ \begin{array}{ccc}
0 & 0 & 1 \\
-am & 1 & 0 \\
-a^{2}  & 0 & 0\\
 \end{array} \right].\\$
\vskip 20pt

Now, recall that  $A^1_1= \left[ \begin{array}{ccc}
1 & 0 & 0 \\
0 & 0 & -a^{-2} \\
0 & 1& -a^{-1}m\\
 \end{array} \right],$ \hskip 50pt $A^1_2= \left[ \begin{array}{ccc}
1 & 0 & -am \\
0 & 0 & -a^2 \\
0 & 1 & 0\\
 \end{array} \right],\\$
 \vskip 20pt
 and,  $B^{-1}_1= \left[ \begin{array}{ccc}
0 & 0 & -a^2 \\
0 & 1 & 0 \\
1 & 0 & -am\\
 \end{array} \right]$,\hskip 50pt $B^{-1}_2= \left[ \begin{array}{ccc}
0 & 0 & -a^{-2} \\
0 & 1 & -a^{-1}m \\
1  & 0 & 0\\
 \end{array} \right].\\$
\vskip 20pt
We note that  each $A^{\pm 1}_i$ is invertible and $A_i^1$ is actually the inverse of $A_i^{-1}$.\\

Also, we note that each $B^{\pm 1}_i$ is invertible and $B_i^1$ is actually the inverse of $B_i^{-1}$.\\

Therefore, $A^{-1}\left[\begin{array}{c}
f(a,m)\\
g(a,m)\\
h(a,m)\\
\end{array} 
  \right]=\left[\begin{array}{c}
0\\
1\\
0\\
\end{array} 
  \right]$ since $P(T_{0})=0\cdot\mathcal{A}+1\cdot\mathcal{B}+0\cdot\mathcal{C}.$\\
  
 This implies that  $(f(a,m)~g(a,m)~h(a,m))=(0~1~0)A^t$.\\
 
 We remark that $\sigma_1$ and $\sigma_2^{-1}$ are corresponding to $\{A_1^1, A_2^1\}$ and $\{B_1^{-1},B_2^{-1}\}$ depending on the given orientation of $w$.\\

 Now, by attaching the three unlinked and unknotted arcs in $B_1$, we can calculate $P(K)=f(a,m)\cdot 1+g(a,m)\cdot({-a-a^{-1}\over m})+h(a,m)\cdot 1$ by Figure~\ref{p8}.\\
 
To see this, we need the fact that The HOMFLY polynomial of a link $L$ that is a split union of two links $L_1$ and $L_2$ is
 given by $P(L)={-a-a^{-1}\over m}P(L_1)P(L_2)$.\\
 
  So, we have $P(S^1 \dot{\cup} S^1)=-({a+a^{-1}\over m})$ for the disjoint union of two unknots.\\
 
Finally, we have $P(K)=f(a,m)-g(a,m)({a+a^{-1}\over m})+h(a,m).$
\\ 
 \begin{figure}[htb]
 \includegraphics[scale=.43]{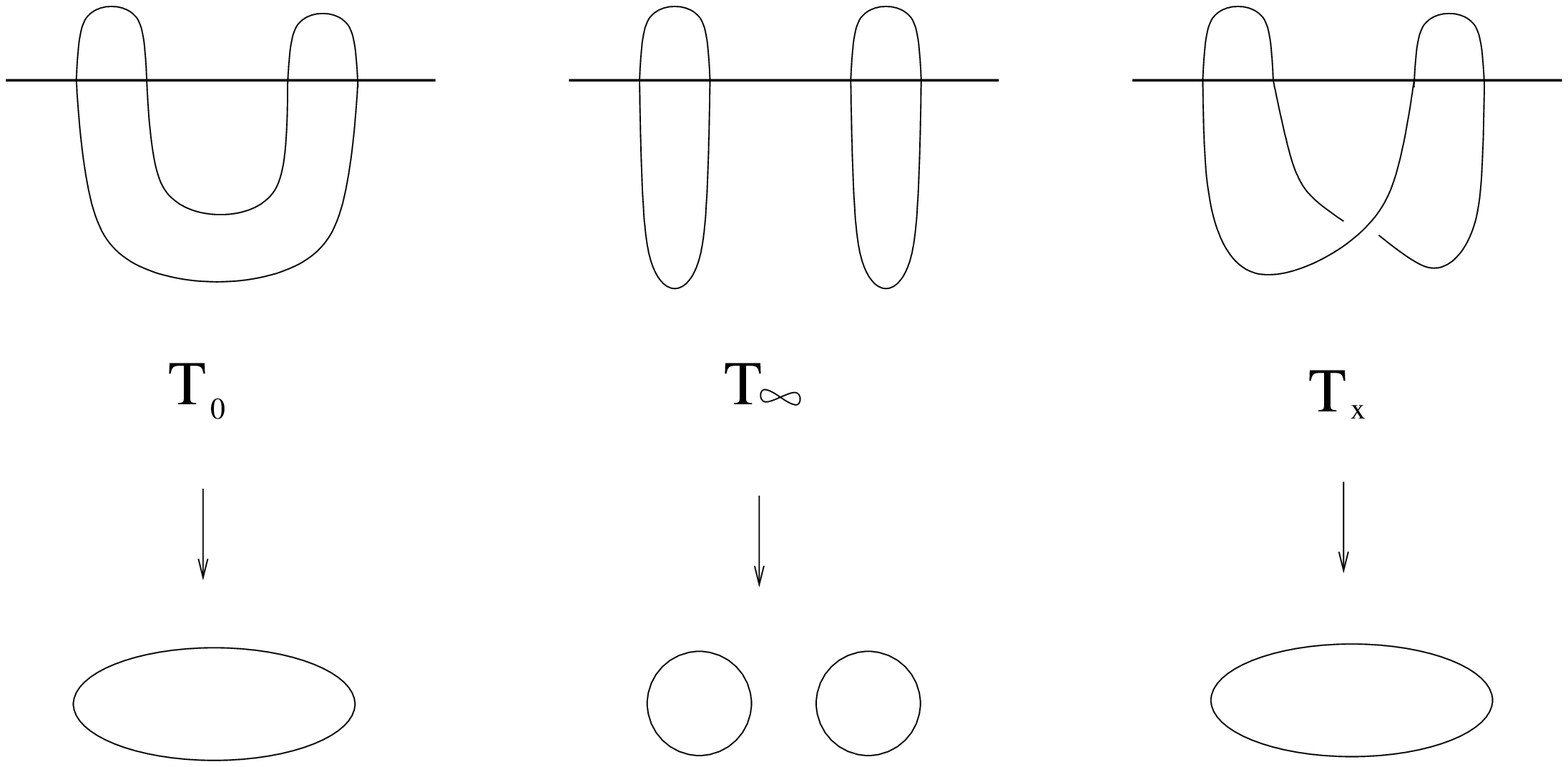}
 \caption{}
 \label{p8}
 \end{figure}
 
\end{proof}

We note that if we have a word $w$ which involves only $\sigma_1^{-1}$ and $\sigma_2$ then we can get a $3\times 3$ matrix $A$
 which is a composing sequences of $\{A_1^{-1},A_2^{-1}\}$ and $\{B_1^1,B_2^1\}$ which are corresponding to $\sigma_1^{-1}$ and $\sigma_2$ depending on a given orientation of $w$.\\
 
 We have the equality about the mirror image of $K$ as follows.\\
 
 $P_K(a,m)=P_{Mirror~Image(K)}(a^{-1},m)$.

\section{A way to determine the $k_i$ for the orientation of a rational three tangle $T$ which has a knot $\overline{T}$}

First, assume that the projection onto the $xy$-plane of a 2-bridge knot $K$ has a standard  plat presentaion $p_4(w)$ with $w=\sigma_1^{\epsilon_1}\sigma_2^{-\epsilon_2}\cdot\cdot\cdot\sigma_1^{\epsilon_{2n-1}}$ for some positive (negative) integers $\epsilon_i$ ($1\leq i\leq 2n-1$).\\

Then we have the  plat presentation $q_4(w)$ of the tangle $T=K\cap B_2$ so that $\overline{q_4(w)}=p_4(w)$.\\ 

Let  $\mathcal{P}(\sigma_i^{\pm 1})$ be the $3\times 3$ matrix which is obtained by interchanging the $i$ and $i+1$ rows of $I$.\\
Then $\mathcal{P}$ extends to a homomorphism from $\mathbb{B}_3$ to $GL_3(\mathbb{Z})$.\\

For an element $w$ of $\mathbb{B}_3$, let 1,2,3 be the upper endpoints of the three strings for $\mathbb{B}_3$  from the left. Also, let $0$ be the upper endpoint of the left most string for $\mathbb{B}_4$. Let $u=[1,2,3]$. Then we assign the same number to the other endpoint of the three strings.  Then we say that the new ordered sequence of numbers $w(u)$ is the $\emph{permutation  induced by $w$}$.

\begin{Lem}\label{T3}
Suppose that $w$ is an element of $\mathbb{B}_3$ so that $w=\sigma_1^{\epsilon_1}\sigma_2^{-\epsilon_2}\cdot\cdot\cdot\sigma_1^{\epsilon_{2n-1}}$ for some positive (negative) integers $\epsilon_i$ ($1\leq i\leq 2n-1$).\\

Then $[1,2,3]\mathcal{P}(w)$ is the permutation which is induced by $w$.
\end{Lem}

\begin{proof}
This is proven by induction on $m=|\epsilon_1|+|\epsilon_2|+\cdot\cdot\cdot+|\epsilon_{2n}|$.

 \end{proof}

 Now, let $[p_w(1),p_w(2),p_w(3)]=[1,2,3]\mathcal{P}(w)$.\\

Without loss of generality, give the orientation (clockwise) to the trivial arc $\delta_1$ in $B_1$ with $\partial\delta_1=\{0,1\}$ from $1$ to $0$ along $\delta_1$. So,  the initial point of $\delta_1$ is 1 and the terminal point of $\delta_1$ is 0 for the given orientation. Then, we can give the orientation to the other trivial arcs $\delta_2$ in $B_1$ as follows, where $\partial \delta_2=\{2,3\}$.\\

Now, we consider $p_w^{-1}(3)$. Then we note that $p_w^{-1}(3)\neq 1$. If not, then $K$ is a link, not a knot.

 \begin{Lem}\label{T4}
 If $p_w^{-1}(3)=3$  then the trivial arc $\delta_2$  has the same direction (clockwise) as $\delta_1$ for the orientation. If $p_w^{-1}(3)=2$   then the trivial arc $\delta_2$  has the opposite direction (counter clockwise) as $\delta_1$ for the orientation.
 \end{Lem}

\begin{proof}
If $p_w^{-1}(3)=3$  then $p_w(3)=3$. So, the direction of the orientation at $3$ is upward. So, the $\delta_2$  has the same direction with $\delta_1$.\\

If $p_w^{-1}(3)=2$  then $p_w(2)=3$.  then the direction of the orientation at $2$ is upward. So, the $\delta_2$  has the opposite direction with $\delta_1$.
\end{proof}

Recall the ordered sequence of numbers $u=[1,2,3]$. Now, we will define a new sequence of numbers $r=[r(1),r(2),r(3)]$.
For the given orientation, we replace the original number for the initial point of $\delta_2$  by 1 as follow.\\

 $r=[r(1),r(2),r(3)]$ so that $r(1)=1$
and for $i>1$,
$r(i)=1$ if $p_w^{-1}(3)=i$ and $r(i)=i$ if $p_w^{-1}(3)\neq i$.\\

For the three strings of the braids $w$, we assign the number $r(k)$ to each string with the upper endpoint $k$ for $1\leq k\leq 3$.\\
 
Now, let  $r_0=r$.\\

 Let $r_i=[r_i(1),r_i(2),r_i(3)]=
 r\mathcal{P}(\sigma_{1}^{\epsilon_1}\sigma_{2}^{-\epsilon_2}\sigma_1^{\epsilon_3}\cdot\cdot\cdot \sigma_{\delta}^{(-1)^{i-1}  \epsilon_{i}})$ for $1\leq i\leq 2n$, where $\delta=1$  if $i$ is odd and $\delta=2$ if $i$ is even.\\

 Let $k_i=\left\{\begin{array}{cl}
 1 & $if $\hskip 10pt  r_{i-1}(2)=r_{i-1}(3) \\
2 &  $if $\hskip 10pt  r_{i-1}(2)\neq r_{i-1}(3). \\
\end{array}\right.$.\\
 
\begin{Thm}\label{T5}
Suppose that the projection of a knot $K$ onto the $xy$-plane has a  plat presentation $p_4(w)$ with $w=\sigma_1^{\epsilon_1}\sigma_2^{-\epsilon_2}\cdot\cdot\cdot\sigma_1^{-\epsilon_{2n-1}}$
for some positive (negative) integers $\epsilon_i$ ($1\leq i\leq 2n-1$).\\

Then a given orientation for $K$, we have $w=\sigma_{1~k_{1}}^{\alpha_1}\sigma_{2~k_{2}}^{-\alpha_2}\cdot\cdot\cdot \sigma_{1~k_{2n-1}}^{\alpha_{2n-1}}$ for $k_i$ which is defined above.
\end{Thm}
 
\begin{proof}
 Without loss of generality, we give the orientation (clockwise) to  $\delta_1$ from $1$ to $2$ along $\delta_1$. Then the direction of the orientation at 1 is up and the direction at 2 is down.
 Then we know that the orientation at $i$ is up if $r(i)=1$ and is down if $r(i)=2$.\\

Fix a value $i$. \\
 
Case 1: Suppose that $r_{i-1}(2)=r_{i-1}(3)=1$.\\
 
 Then the two strings for the ($\sum_{j=1}^{i-1}|\epsilon_j|+1)$-th crossing have the same direction  of the orientation since  $r_{i-1}(2)=r_{i-1}(3)=1$ So, the directions of the orientation are upward.\\
 
 Then  $k_i=1$. This is consistant with the $k_i$ that is defined above.\\

  Case 2: Suppose that $r_{i-1}(2)\neq r_{i-1}(3)$.\\
 
 Then the two strings for the ($\sum_{j=1}^{i-1}|\epsilon_j|+1)$-th crossing have different directions  for the orientation since  $r_{i-1}(2)\neq r_{i-1}(3)$.\\
 
 Then  $k_i=2$. This is consistant with the $k_i$ that is defined above.\\

\end{proof}

\section{The calculation of some examples}

First, we will calculate the HOMFLY polynomials of $3_1$ (trefoil knot), $5_1$ and $5_2$.\\

\begin{figure}[htb]
\includegraphics[scale=.4]{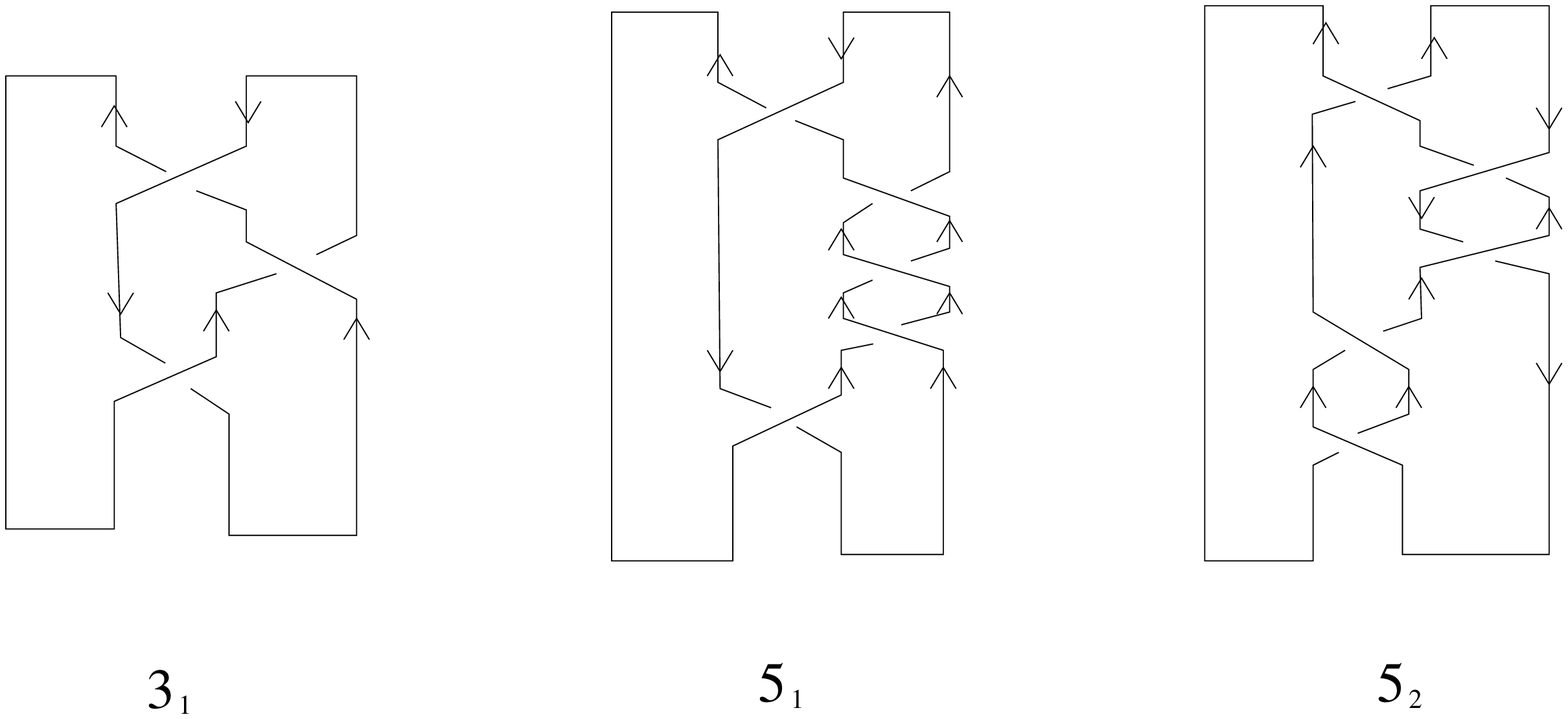}
\caption{}
\label{p9}
\end{figure}

\begin{enumerate}
\item[(a)] $3_1$ is represented by $w=\sigma_{12}\sigma_{21}^{-1}\sigma_{12}$. Then we have $A=A_2^1B_1^{-1}A_2^1$.\\

So, $[f(a,m)~g(a,m)~h(a,m)]=[-a^2+a^2m^2~a^3m~0]$.\\

Therefore, $P(3_1)=-a^2+a^2m^2-a^3m{a+a^{-1}\over m}=-a^2+a^2m^2-a^4-a^2=-2a^2+a^2m^2-a^4$.\\

\item[(b)] $5_1$ is represented by $w=\sigma_{12}\sigma_{21}^{-3}\sigma_{12}$. Then we have $A=A_2^1(B_1^{-1})^3A_2^1$.\\

So, $[f(a,m)~g(a,m)~h(a,m)]=[a^4-3a^4m^2+a^4m^4~a^5m(-2+m^2)~0]$.\\

Therefore, $P(5_1)=-a^6m^2+2a^6+a^4m^4-4a^4m^2+3a^4$.\\

\item[(c)] $5_2$ is represented by $w=\sigma_{11}^{-1}\sigma_{22}^{2}\sigma_{11}^{-2}$. Then we have $A=A_1^{-1}(B_2^{1})^2A_1^{-2}$.\\

So, $[f(a,m)~g(a,m)~h(a,m)]=[a^3m-a^3m^3-a^5m+a^5m^3~ a^4-a^4m^2+a^6m^2~0]$.\\

Therefore, $P(5_2)=a^6-a^2+a^2m^2+a^4-a^4m^2$.\\

Now, I will give a set of knots $K_1$  and $K_2$ so that they have the same Jones polynomial but different HOMFLY polynomials.\\

Consider the two knots $K_1=8_9$ and $K_2=4_1\# 4_1$.\\

First, we can check that the Kauffman (Jones) polynomial of $K_1$ and $K_2$ are the same as follows.\\

 $X_{8_9}=X_{4_1\#4_1}=(a^{16}-a^{12}+a^8-a^4+1)^2/a^{16}$.

\begin{figure}[htb]
\includegraphics[scale=.4]{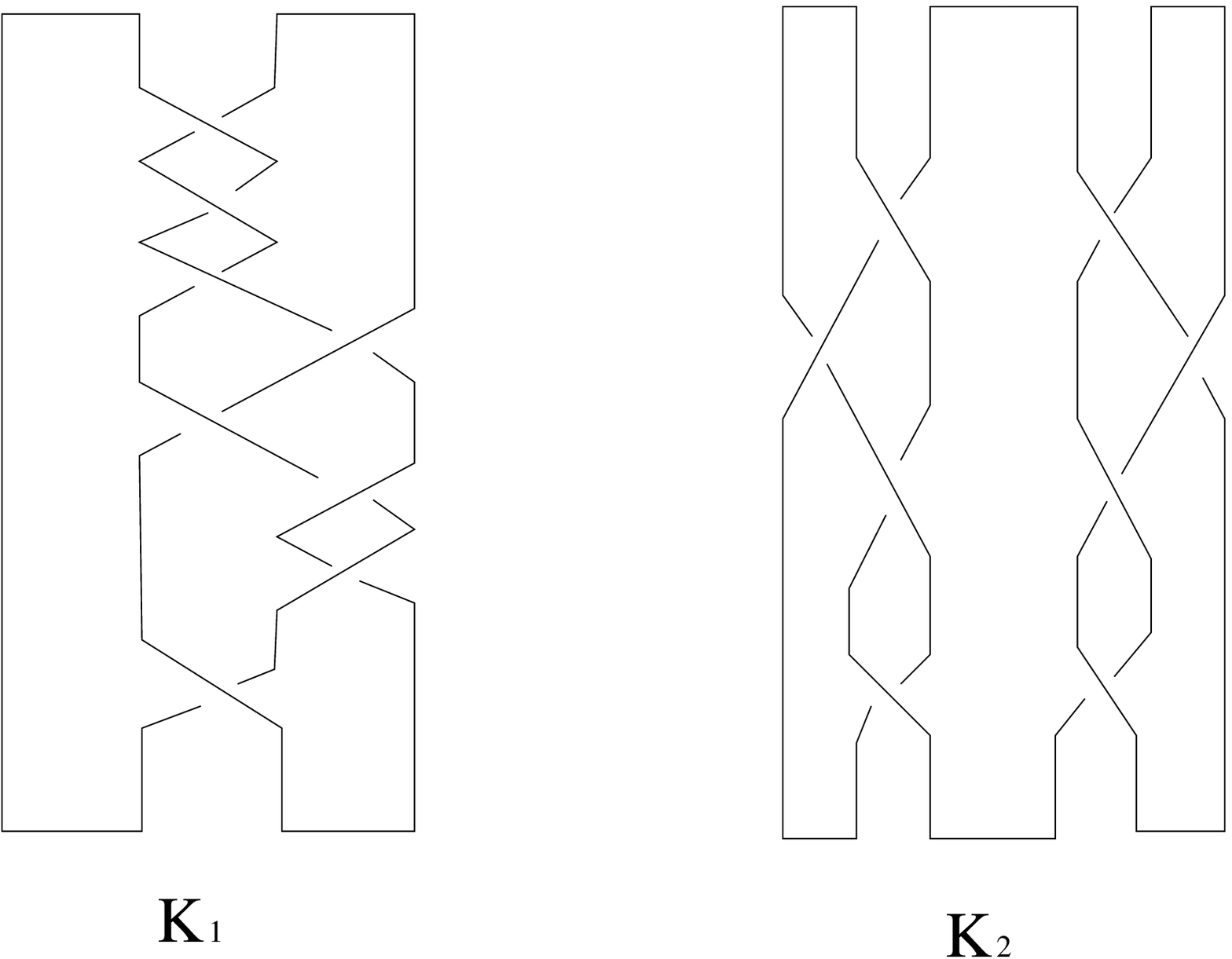}
\caption{}
\label{p10}
\end{figure}
\end{enumerate}

We note that $8_9$ is represented by $w=(\sigma_{11}^{-1})^3(\sigma_{22})(\sigma_{12}^{-1})(\sigma_{21})^2(\sigma_{12}^{-1})$.\\

So, $[f(a,m)~g(a,m)~h(a,m)]=[0~(-m^7a+4am^5-5am^3+a^{-1}m^5-2a^{-1}m^3+a^{-1}m+a^4a^{-1}m^5-3a^3m^3+2a^3m+am)~(-m^6a^2+3a^2m^4+m^4-3a^2m^2-m^2+a^4m^4-2a^4m^2+a^4)].$\\

Therefore, $P(8_9)=-2a^2m^4+5a^2m^2-4m^4+6m^2-2+a^4m^2-a^4-3a^2+m^6-a^{-2}m^4+2a^{-2}m^2-a^{-2}$.\\

However, $4_1$ is prepresented by $w=(\sigma_{11}^{-1})(\sigma_{22})(\sigma_{12}^{-1})^2$.\\

So, $[f(a,m)~g(a,m)~h(a,m)]=[0~(-am^3+a^{-1}m+a^3m)~(-a^2m^2+a^4)]$.\\

Therefore, $P(4_1)=m^2-a^2-1-a^{-2}$.\\

This implies that $P(4_1\#4_1)=P(4_1)\cdot P(4_1)=(m^2-a^2-1-a^{-2})^2$.\\

We can check that $P(8_9)\neq P(4_1\#4_1)$.\\

\begin{figure}[htb]
\includegraphics[scale=.4]{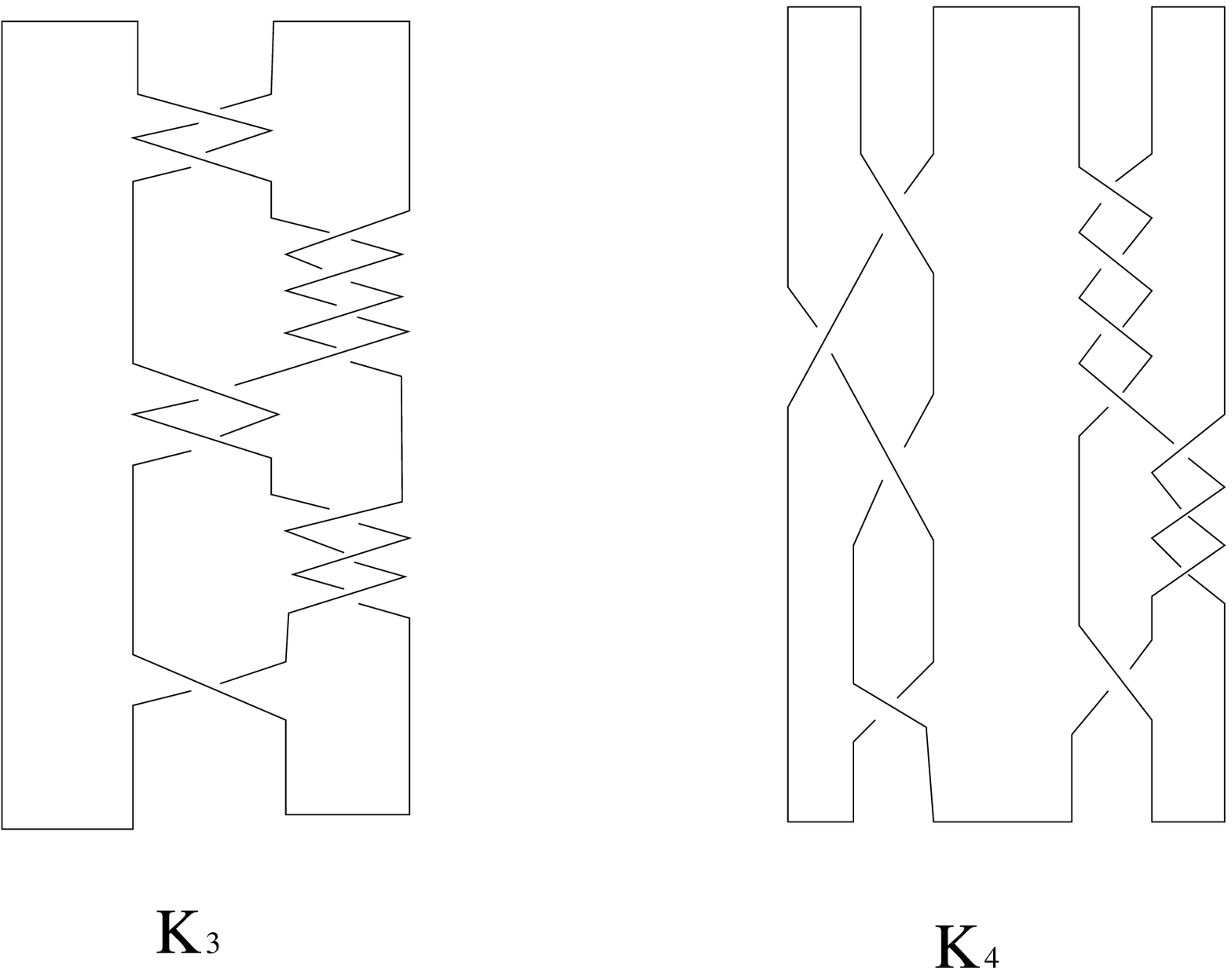}
\caption{}
\label{p11}
\end{figure}

Now, consider the two knots $K_3$ and $K_4=4_1\#8_3$ as in Figure~\ref{p11}.\\

We can check that they have the same Kauffman (Jones) polynomial as follows.\\

$X_{K_3}=X_{K_4}=(a^{16}-a^{12}+a^8-a^4+1)(a^{32}-a^{28}+2a^{24}-3a^{20}+3a^{16}-3a^{12}+2a^8-a^4+1)/a^{24}$.\\

Then $K_3$ is represented by $w=(\sigma_{12}^{-1})^2(\sigma_{22}^1)^4(\sigma_{12}^{-1})^2(\sigma_{22}^1)^3(\sigma_{11}^{-1})$.\\

So, $[f(a,m)~g(a,m)~h(a,m)]=[(a^{10}-2a^8m^2+2a^6m^2+m^4a^6-2a^4m^4+a^2m^4+a^2m^2-m^2)/a^2,-m(a^2-1)(a^6-a^4m^2+a^2m^2-1)/a^3,0]$.\\

Therefore, $P(K_3)=(a^{12}-2a^{10}m^2+a^8m^2+a^8m^4-2m^4a^6+a^4m^4+2a^4m^2-2a^2m^2+a^{10}-a^6+a^6m^2-a^4+1)/a^4$.\\

$8_3$ is represented by $w=(\sigma_{12}^{-1})^4(\sigma_{22}^1)^3(\sigma_{11}^{-1})$.\\

So,  $[f(a,m)~g(a,m)~h(a,m)]=[(a^6-a^4m^2+2a^2m^2-m^2)/a^2,m(a^2-1)/a^3,0]$.\\

Therefore, $P(8_3)=(a^8-a^6m^2+2a^4m^2-a^2m^2-a^4+1)/a^4$.\\

Now, we know that $P(4_1\#8_3)=(m^2-a^2-1-a^{-2})(a^8-a^6m^2+2a^4m^2-a^2m^2-a^4+1)/a^4$.\\

This implies that $P(K_3)\neq P(K_4)=P(4_1\#8_3)$.

\end{document}